\documentclass[a4paper]{article}
\usepackage[margin=0.7in]{geometry} 

\usepackage{amssymb,amsmath,amsfonts,amsthm,graphicx,verbatim,enumerate,enumitem}
\usepackage{hyperref}
\usepackage{xargs}                      
\usepackage[pdftex,dvipsnames]{xcolor}  
\usepackage[colorinlistoftodos,prependcaption,textsize=tiny]{todonotes}
\newcommandx{\unsure}[2][1=]{\todo[linecolor=red,backgroundcolor=red!25,bordercolor=red,#1]{#2}}
\newcommandx{\change}[2][1=]{\todo[linecolor=blue,backgroundcolor=blue!25,bordercolor=blue,#1]{#2}}
\newcommandx{\info}[2][1=]{\todo[linecolor=OliveGreen,backgroundcolor=OliveGreen!25,bordercolor=OliveGreen,#1]{#2}}
\newcommandx{\improvement}[2][1=]{\todo[linecolor=Plum,backgroundcolor=Plum!25,bordercolor=Plum,#1]{#2}}
\newcommandx{\thiswillnotshow}[2][1=]{\todo[disable,#1]{#2}}

\theoremstyle{plain}
\newtheorem{proposition}{Proposition}[section]
\newtheorem{theorem}[proposition]{Theorem}
\newtheorem{corollary}[proposition]{Corollary}

\newtheorem{definition}[proposition]{Definition}

\newtheorem{lemma}[proposition]{Lemma}
\newtheorem{conjecture}[proposition]{Conjecture}
\newtheorem{problem}[proposition]{Problem}
\newtheorem{prerk}[proposition]{Remark}
\newenvironment{remark}{\prerk\rm}{\endprerk}

\def\mc{\mathcal}
\def\mb{\mathbf}
\DeclareMathOperator{\gn}{gn}
\DeclareMathOperator{\Win}{Win}
\DeclareMathOperator{\rk}{rk}
\DeclareMathOperator{\Ran}{Random}
\DeclareMathOperator{\F}{\mathbb{F}}

\begin{document}

\title{Guessing Games on Triangle-free Graphs}

\author{Peter J. Cameron
\thanks{School of Mathematical Sciences,
Queen Mary, University of London, London, E1 4NS, U.K. Current address:
School of Mathematics and Statistics, University of St Andrews, North Haugh, St Andrews, KY16 9SS, U.K. Email: pjc20@st-andrews.ac.uk.}
\and Anh N. Dang
\thanks{School of Electronic Engineering and Computer Science,
Queen Mary, University of London, London, E1 4NS, U.K.}
\and S{\o}ren Riis
\thanks{School of Electronic Engineering and Computer Science,
Queen Mary, University of London, London, E1 4NS, U.K. Email:
s.riis@qmul.ac.uk}
}

\date{}

\maketitle

\begin{abstract}
The guessing game introduced by Riis\cite{Riis07} is a variant of the ``guessing your own hats'' game and can be played on any simple directed graph $G$ on $n$ vertices. For each digraph $G$, it is proved that there exists a unique guessing number $\gn(G)$ associated to the guessing game played on $G$. When we consider the directed edge to be bidirected, in other words, the graph $G$ is undirected, Christofides and Markstr\"{o}m \cite{Christofides&Markstrom11} introduced a method to bound the value of the guessing number from below using the fractional clique cover number $\kappa_f(G)$. In particular they showed $\gn(G) \geq |V(G)| - \kappa_f(G)$. Moreover, it is pointed out that equality holds in this bound if the underlying undirected graph  $G$ falls into one of the following categories: perfect graphs, cycle graphs or their complement. In this paper, we show that there are triangle-free graphs that have guessing numbers which do not meet the fractional clique cover bound. In particular, the famous triangle-free Higman--Sims graph has guessing number at least $77$ and at most $78$, while the bound given by fractional clique cover is $50$.
\end{abstract}

\section{Introduction}
The motivation of developing guessing games \cite{Riis07} comes from the study of a specific class of problems in network coding \cite{AhlswedeCLY00}, namely multiple unicast network coding. A multiple unicast network is a communication network in which each sender has a unique receiver that wishes to obtain messages from it. Such a network can be represented by a directed acyclic graph where senders, receivers and routers are vertices in the graph and channels are directed edges between vertices. Furthermore, we simplify the problem by require that each channel only allows one message to pass through it at a time. If we merge each vertex representing a sender with a vertex which represent its corresponding receiver in the directed acyclic graph, we will obtain an auxiliary digraph in which we no longer have the distinction between senders, receivers, or routers. We can define a guessing game to play on this auxiliary digraph; the rules of our guessing game will allow us to translate strategies on the auxiliary graph into coding functions on the simplified multiple unicast network and \emph{vice versa}. The guessing number is a measure of the performance of the optimal strategy for a guessing game; its precise definition will be given in Section \ref{Sec:Definitions}.

It is worth noting that guessing games were the main ingredients in Riis' proof of the invalidity of two conjectures raised by Valiant~\cite{Valiant} in circuit complexity in which he asked about the optimal Boolean circuit for a Boolean function. (See \cite{Riis07}.)

Our paper deals with the guessing game played on a special type of directed graph in which each directed edge is bidirected, i.e. our graphs are undirected. In particular, we show that there are triangle-free undirected graphs where the guessing numbers of these graphs can not be computed using the fractional clique cover method developed by Christofides and Markstr\"om in \cite{Christofides&Markstrom11}. This also gives counterexamples to their conjecture about the optimal guessing strategy based on fractional clique cover in \cite{Christofides&Markstrom11}. (The first counterexample to Christofides and Markstr\"om's conjecture was illustrated in \cite{Baber13} but the graph is not triangle-free.)

Our paper is organised as as follows. Firstly, we introduce the rules of guessing games played on undirected graphs in Section \ref{Sec:Definitions}. Then in Section \ref{Sec:Asymptotic} we prove the existence of the asymptotic guessing numbers. In Section \ref{Sec:FractionalCliqueCover} the fractional clique cover strategy from \cite{Christofides&Markstrom11} is formally defined. Our main results appear in Section \ref{Sec:TriangleFreeGraphs}. Sections \ref{Sec:Definitions}, \ref{Sec:Asymptotic}, \ref{Sec:FractionalCliqueCover} already appeared in \cite{GR11}, \cite{Christofides&Markstrom11}, and \cite{Baber13} but we reproduce them here in order to make this paper self-contained.

\section{Definitions}\label{Sec:Definitions}
An \emph{undirected graph} $G = (V,E)$, or \emph{graph} for short, consists of a set $V(G)$ of \emph{vertices} and a set $E(G)$ of \emph{undirected edges}. An undirected edge $e \in E(G)$, or edge for short, is an unordered pair $(u, v)$ of vertices, which we also denote by $uv$ or $vu$, with $u$ and $v$ are elements of $V(G)$. We say vertices $u$ and $v$ are \emph{adjacent} if $uv$ is an edge. Similarly, we say two edges are adjacent if they share a common vertex. Given a graph $G$, we will denote its adjacency matrix by $A$. We also denote the identity matrix by $I$, and it will be clear from the context that $I$ and $A$ have the same order.

The guessing game is defined to be played on \emph{simple graphs}, i.e.\ graphs not containing \emph{loops} (edges of the form $uu$ for $u$ a vertex) or \emph{multiple edges} (two or more edges with the same vertices). Thus, two edges are adjacent if and only if they share exactly one common vertex.

Given a graph $G$ and a vertex $v\in V(G)$, the \emph{neighbourhood} of $v$ is $\Gamma(v) = \{u : uv \in E(G)\}$.

An important class of guessing strategies introduced in Section \ref{Sec:FractionalCliqueCover} involve \emph{cliques}, i.e.\ subgraphs of $G$ in which every pair of vertices are joined by an undirected edge.

Given two graphs $G$ and $H$, the \emph{tensor product} or \emph{categorical product} $G \times H$ of $G$ and $H$ is a graph with vertex set $V(G \times H)$ which is the Cartesian product of $V(G)$ and $V(H)$; and an edge $e = ((u, v), (u', v')) \in E(G \times H)$ if $uu'$ is an edge in $G$ and $vv'$ is an edge in $H$. A special case where $H$ is $K_t^\circ$ a complete graph of order $t$ with a loop at each vertex, i.e.\ $|V(H)|=t$ and $E(H)$ is the set of all unordered pairs $(u, v)$ with $u, v\in V(G)$, we will call $G \times H$ the \emph{$t$-uniform blowup} of $G$, and denote as $G(t)$. We can also see the $t$-uniform blowup as a graph formed by replacing each vertex $v$ in $G$ with a class of $t$ vertices $v_1,\ldots,v_t$ with $u_iv_j\in E(G(t))$ if and only if $uv \in E(G)$.

Given a graph $G=(V,E)$ of order $n$, and a positive integer $s$ greater than $1$, we play a \emph{guessing game} $(G,s)$ as follows:

A $(G,s)$ game consists of $n$ players with each player corresponding to one of the vertices of $G$. Throughout this paper, we use vertex $v \in V(G)$ to indicate the player who is assigned to vertex $v$. Each player is informed about its corresponding vertex, its neighbourhood $\Gamma(v)$, and the natural number $s$. The players can use this information to decide a strategy beforehand, but all communication are forbidden as soon as the game is started.

Once the game starts each player $v\in V(G)$ is assigned a value $a_v$ from \emph{alphabet set} $A_s = \{0, 1, \ldots, s-1\}$ uniformly and independently at random. The value $a_v$ is hidden from the player $v$. Instead, each player $v$ is provided a list containing its neighbourhood $\Gamma(v)$ and the value assigned to each of its vertices. The player is required to deduce its own assigned value using just this information. Each player must announce its guessed value only once. A game is won if every player deduces the assigned value correctly, and the game is lost otherwise. We are interested in the question about the maximal probability of winning when we play a guessing game $(G, s)$.

To illustrate the nature of this problem, let us play guessing games $(C_5, s)$ with $s = 2, 3, 4$ and $C_5$ the cycle graph of order $5$. This is example 3.2 in \cite{Christofides&Markstrom11}. Naively each player should guess randomly as none of the provided information directly relates to its assigned value, hence the winning probability of this strategy is $s^{-5}$. We name this naive strategy $\Ran$. The interesting property of guessing game is that we are almost always able to outperform the $\Ran$ strategy.

For $s = 2$, a possible strategy for the game $(C_5, 2)$ is as follows: each player $v$ guesses $0$ if all of its neighbours are assigned $1$, and guesses $1$ for the rest cases. The game is won if the assigned values to players belong to one of the following cases $\{11010, 10110, 10101, 01101, 01011\}$. Hence, the winning probability with this strategy is $\frac{5}{2^5} = \frac{5}{32}$, and indeed this is the highest possible winning probability for the guessing game $(C_5, 2)$ \cite{Christofides&Markstrom11}.

For $s = 3$, using computer search, Christofides and Markstr\"om \cite{Christofides&Markstrom11} were able to show that the best possible winning probability for $(C_5, 3)$ is $\frac{12}{3^5} = \frac{12}{243}$ with a complicated guessing strategy which is highly non-symmetric in the sense that the vertices all use a different guessing functions.

For $s=4$, it is shown in \cite{ Riis07B} that an optimal strategy involves the so-called \textit{fractional clique cover strategy} which we will introduce in Section \ref{Sec:FractionalCliqueCover}. The winning probability corresponds to this strategy is $\frac{2*4^2}{4^5} = 4^{-2.5}$.

In the examples above, guessing strategies are \emph{pure strategies}, i.e.\ each player deduces its guessing value by using a deterministic function whose inputs are the values of its neighbours. We can also play a \emph{mixed strategy} in which the players randomly select a strategy to follow from a set of pure strategies. The winning probability of a mixed strategy is the expected winning probability computed by averaging the winning probabilities of the chosen pure strategies using the probabilities that they are selected. Hence, playing with mixed strategy will gain us no advantage when computing maximal winning probabilities as the maximal winning probability with a mixed strategy cannot surpass the winning probability provided by the best pure strategies. Therefore we only be concerned with pure strategies throughout this paper.

Let $(G, s)$ be a guessing game played on the graph $G$ with alphabet set $A_s = \{0, 1, \ldots, s-1\}$ with $s$ is a positive integer greater than $1$. Each \emph{strategy for player $v$} is a function $f_v: A_s^{|\Gamma(v)|} \to A_s$ which takes the possible values of the neighbours of $v$ and maps them to the guessing value of $v$. A \emph{strategy} $\mc{F}$ for $(G, s)$ is a $|V(G)|-$tuple of deterministic functions $(f_v)_{v\in V(G)}$ where $f_v$ corresponds to a strategy for player $v$. We denote $\Win(G, s, \mc{F})$ for the event that the game $(G, s)$ is won by using strategy $\mc{F}$. We are interested in constructing a strategy $\mc{F}$ which maximises $\mb{P}[\Win(G, s, \mc{F})]$.

For each strategy $\mc{F}$, we denote $\gn(G, s, \mc{F})$ for the value $|V(G)|+\log_s\mb{P}[\Win(G,s,\mc{F})]$ which is the \emph{guessing number with respected to $\mc{F}$}. We define the guessing number $\gn(G, s)$ to be
\[
\gn(G,s) = |V(G)|+\log_s
\left(\max_{\mc{F}}\mb{P}[\Win(G,s,\mc{F})]\right).
\]

We can see that:
\[
\max_{\mc{F}}\mb{P}[\Win(G,s,\mc{F})] =
\frac{s^{\gn(G,s)}}{s^{|V(G)|}}.
\]
The guessing number $\gn(G, s)$ is a measure of how much better an optimal strategy outperforms the random strategy $\Ran$ when playing $(G, s)$.

\section{The asymptotic guessing number}\label{Sec:Asymptotic}
In our examples of guessing games $(C_5,s)_{s = 2, 3, 4}$, the guessing numbers $\gn(C_5,s)$ depend on $s$, and in this case the sequence $\{\gn(C_5,s)\}_{s = 2,3,\ldots}$ is not a monotone sequence of $s$. In general case for guessing games $(G, s)$, it is extremely difficult to determine the exact value of $\gn(G, s)$ for each value of $s$. Therefore, we rather interested in evaluating the value of $\gn(G, s)$ when $s$ tends to infinity, and we call it the \emph{asymptotic guessing number} $\gn(G)$:
\[
\gn(G) = \lim_{s \to \infty} \gn(G, s).
\]
It is proved in \cite{Christofides&Markstrom11} and \cite{GR11} that this limit exists. The following arguments are due to Christofides and Markstr\"om \cite{Christofides&Markstrom11}. Their strategy is to prove that the sequence  $\{\gn(G,s)\}_{s = \{2, 3, \ldots\}}$ for a general graph $G$ is an almost monotonically increasing sequence with respect to the size $s$ of the alphabet $A_s$, and this sequence is bounded above by the obvious bound $|V(G)|$ (in fact in their paper \cite{Christofides&Markstrom11}, the upper bound of this sequence is $|V(G)| - \alpha(G)$ where $\alpha(G)$ is the independence number of the graph $G$).

We start with the following lemma.

\begin{lemma}\label{Lem:BlowupIneq}
Let $G$ be an undirected graph, and $s, t$ integers with $s \geq 2$ and $t \geq 1$. We have
\begin{equation}\label{Eq:Blowup1}
t \gn(G, s) \leq \gn(G(t), s)
\end{equation}
\begin{equation}\label{Eq:Blowup2}
 \gn(G(t), s) = t \gn(G, s^t)
\end{equation}
\end{lemma}
\begin{proof}
The graph $G(t)$ contains $t$ vertex disjoint copies of $G$. Given a strategy $\mc{F}$ played on $(G,s)$, we construct a strategy $\mc{F}(t)$ of $(G(t), s)$ by playing $\mc{F}$ on each of the $t$ disjoint copies of $G$ in $G(t)$. This gives us
\[
(\max_{\mc{F}}\mb{P}[\Win(G,s,\mc{F})])^t \leq \max_{\mc{F}}\mb{P}[\Win(G(t),s,\mc{F})]
\]
and the inequality (\ref{Eq:Blowup1}) follows immediately.

To prove (\ref{Eq:Blowup2}) holds, we show that there is an one-to-one correspondence between strategies played on $(G(t),s)$ and strategies played on $(G, s^t)$, with the property that given a strategy $\mc{F}(t)$ played on $(G(t),s)$ its corresponding strategy $\mc{F}$ played on $(G,s^t)$ give the same winning probabilities. Hence, we have
\[
\max_{\mc{F}}\mb{P}[\Win(G(t),s,\mc{F})] =
\max_{\mc{F}}\mb{P}[\Win(G,s^t,\mc{F})].
\]
and the result follows from the definition of guessing number.

We note that each member $a$ of the alphabet of size $s^t$ can be uniquely represented as a $t$-tuple $\{a_1, \ldots, a_t\}$ with $a_i$s are in base $s$.

Given a strategy $\mc{F}$ on $(G,s^t)$, we construct a corresponding strategy $\mc{F}(t)$ to be played on $(G(t),s)$ as follow:

For each player $v$ in $(G,s^t)$, its guessing function is $f_v: A_{s^t}^{|\Gamma(v)|} \to A_{s^t}$. The map $f_v$ induces a unique corresponding map $f_{[v_1, \ldots, v_t]}: (A_s^{t})^{|\Gamma(v)|^t} \to A_s^{t}$ due to the one-to-one correspondence between elements of $A_{s^t}$ and ${A_s}^t$. Therefore, if a player $v$ in $(G,s^t)$ follows a strategy $f_v$, the vertex class of $t$ players $[v_1, \ldots, v_t]$ in $(G(t),s)$ simulate playing as $v$ by agreeing to use $f_{[v_1, \ldots, v_t]}$ as guessing function for each member in the class. The output of $f_{[v_1, \ldots, v_t]}$ can be considered as the guess of each member $v_i$ about the overall value assigned to the whole class $[v_1, \ldots, v_t]$. This guessing function is well defined since each $v_i$ receives precisely the same input. Moreover, the guessing output produced by each member $v_i$ of the class will be the same. Each member can decompose the value of the guess for the vertex class into $t$ values from $A_s$ and uses this information as the individual guesses for each one of them. Hence we have converted a guessing strategy $\mc{F}$ of $(G,s^t)$
to a guessing strategy $\mc{F}(t)$ of $(G(t),s)$. We can also see that
\[
\mb{P}[\Win(G,s^t,\mc{F})] = \mb{P}[\Win(G(t),s,\mc{F}(t))].
\]

Clearly the map $\mc{F} \mapsto \mc{F}(t)$ defined above is an injection from the set of all guessing strategies can be played on $(G,s^t)$ to the set of all guessing strategies can be played on $(G(t),s)$. We use a similar argument to show that given a guessing strategy $\mc{F}(t)$ to be played on $(G(t),s)$ there is a corresponding guessing strategy $\mc{F}$ of $(G,s^t)$.

The unique correspondence between $f_v: A_{s^t}^{|\Gamma(v)|} \to A_{s^t}$ and $f_{[v_1, \ldots, v_t]}: (A_s^{t})^{|\Gamma(v)|^t} \to A_s^{t}$ described above allows each player $v$ in $(G, s^t)$ to pretend to be $t$ players in the class $[v_1, \ldots, v_t]$ in $(G(t),s)$. The strategy $\mc{F}(t)$ of $(G(t),s)$ can then be used and the guesses for each vertex class $[v_1, \ldots, v_t]$ can be reconstructed into a guess for the original player $v$ in $(G,s^t)$. This completes the proof.
\end{proof}

Using these results about the guessing number of the $t$-uniform blowup of graphs we now show that the guessing number is almost monotonically increasing with respect to the size of the alphabet.

\begin{lemma}\label{Lem:Monotonic}
Given a graph $G$, positive integer $s$, and real number $\epsilon>0$, there exists $t_0(G, s,\epsilon)>0$ such that for all integers $t\geq t_0$
\[\gn(G,t)\geq \gn(G,s)-\epsilon.\]
\end{lemma}

\begin{proof}
To prove the statement holds, it is sufficient to show that
\begin{align}\label{Eq:AlmostIncLog}
\gn(G,t)\geq\frac{\lfloor \log_st\rfloor}{\log_st}\gn(G,s)
\end{align}
holds for all $t\geq s$ since the right hand side of (\ref{Eq:AlmostIncLog}) tends to $\gn(G,s)$ as $t$ increases.

Let $k=\lfloor \log_st\rfloor$. On the set of all guessing strategies of $(G,t)$, we consider only strategies $\mc{F}= (f_v)_{v \in V(G)}$ such that $f_v$ is a map $A_t^{\Gamma(v)} \to \{0, 1, \ldots, s^k - 1\}$ for every $v \in V(G)$. The maps $f_v$s are well defined as $t \geq s^k$. We have
\[
\max_{\mc{F}}\mb{P}[\Win(G,t,\mc{F})] \geq \mb{P}[a_v<s^k \mbox{ for
all } v\in V(G)]\max_{\mc{F}}\mb{P}[\Win(G,s^k,\mc{F})].
\]
Hence
\begin{align}\label{Eq:AlmostIncreasingPow}
\frac{t^{\gn(G,t)}}{t^{|V(G)|}}\geq
\left(\frac{s^k}{t}\right)^{|V(G)|}
\frac{s^{k\gn(G,s^k)}}{s^{k|V(G)|}}.
\end{align}
Rearranging (\ref{Eq:AlmostIncreasingPow}), we have
\begin{align}\label{Eq:AlmostIncreasing}
\gn(G,t)\geq\frac{k}{\log_st}\gn(G,s^k).
\end{align}
Lemma \ref{Lem:BlowupIneq} shows that $\gn(G,s^k)\geq\gn(G,s)$ which, together with (\ref{Eq:AlmostIncreasing}), completes the proof of
(\ref{Eq:AlmostIncLog}).
\end{proof}

\begin{theorem}
For any graph $G$, $\gn(G) = \lim_{s\to\infty} \gn(G,s)$ exists.
\end{theorem}

\begin{proof}
The sequence $\{\max_{s \leq n}\gn(G,s)\}_{n = \{2, 3, \ldots\}}$ is an increasing sequence, and by definition $\gn(G,s)\leq |V(G)|$ for all $s$, therefore the limit $\lim_{n \to \infty} \max_{s \leq n}\gn(G,s) =: \gn(G)$ exists. We note that $\gn(G,s)\leq \gn(G)$ for all $s$. It is sufficient to show that $\gn(G,s)$ converges to $\gn(G)$ from below.

Given $\epsilon>0$ there exists $s_0(\epsilon)$ such that $\gn(G,s_0(\epsilon))\geq \gn(G)-\epsilon$ (by the definition of $\gn(G)$). Lemma \ref{Lem:Monotonic} proves that there exists $t_0(\epsilon)$ satisfying the condition for all $t\geq t_0(\epsilon)$, $\gn(G,t)\geq \gn(G,s_0(\epsilon))-\epsilon$; this implies $\gn(G,t)\geq \gn(G) - 2\epsilon$, proving the convergence.
\end{proof}

\begin{remark}
A consequence of Lemma \ref{Lem:Monotonic} is that for any $s$ the guessing number $\gn(G,s)$ is a lower bound for $\gn(G)$.
By definition of guessing number $\gn(G,s)$, we have
\[
\gn(G,s) \geq |V(G)|+\log_s \mb{P}[\Win(G,s,\mc{F})].
\]
for any strategy $\mc{F}$ on $(G,s)$.
Therefore, any strategy on any alphabet size provides us a lower-bound for the asymptotic guessing number.
\end{remark}

\section{Lower bounds using the fractional clique cover}\label{Sec:FractionalCliqueCover}
The remark in the previous section tells us that we can provide a lower bound for the asymptotic guessing number using any guessing strategy on an alphabet of any size. Christofides and Markstr\"om \cite{Christofides&Markstrom11} used this fact to provide a simple lower-bound for guessing number $\gn(G)$ by constructing a general guessing strategy for graphs $G$ called the fractional clique cover strategy.

Given a graph $G$, we denote $K(G)$ for the set of all cliques in $G$, and denote $K(G,v)$ for the set of all cliques in $G$ containing vertex $v$. A \emph{fractional clique cover} of $G$ is a weighting $w:K(G)\to [0,1]$ such that
\[
\sum_{k\in K(G,v)} w(k) \geq 1.
\]
for all $v\in V(G)$.
We denote $\kappa_f(G)$ for the minimum value of $\sum_{k\in K(G)}w(k)$ over all choices of fractional clique covers $w$. (Note for any graph $G$, we have the identity: $\kappa_f(G) = \chi_f(G^c)$ where $\chi_f(G^c)$ is the fractional chromatic number of the graph's complement.)

We say a fractional clique cover is \emph{regular} if its weighting $w:K(G)\to [0,1]$ satisfies
\[
\sum_{k\in K(G,v)} w(k) = 1.
\]
for all $v\in V(G)$.

We will only consider regular fractional clique covers from now on, as they are more convenient for our purpose of constructing guessing game strategies. In fact, even though we only focus on this smaller class of fractional clique cover, we do not lose any information about the value of $\kappa_f(G)$ as it can be proved that
\[\kappa_f(G) = \min_{w\text{ regular}} \sum_{k\in K(G)}w(k)\]

To prove the above identity holds, we need to show that
\[\kappa_f(G) \leq \min_{w\text{ regular}} \sum_{k\in K(G)}w(k)\]
and
\[\kappa_f(G) \geq \min_{w\text{ regular}} \sum_{k\in K(G)}w(k)\]
The first inequality comes from the definition of $\kappa_f(G)$. To show the second inequality holds, we prove that given an optimal fractional clique cover $w$ we can make it into a regular fractional cover $w_r$ with a property that $\sum_{k\in K(G,v)} w(k) = \sum_{k\in K(G,v)} w_r(k)$. This is done by moving weights from larger cliques to smaller ones.

Let $k \in K(G,v)$ be a clique containing $v$. We denote $k' = k\backslash\{v\}$ be the subclique obtained by removing vertex $v$. If we reduce the weight $w(k)$ and increase the weight $w(k')$ by the same amount then the sum $\sum_{k\in K(G,v)} w(k)$ is reduced but all other sums remain constant. Hence, the result follows.

The result of Christofides and Markstr\"om states the following:

\begin{theorem}\label{Thm:LowerBoundIneq}
If $G$ is an undirected graph then
\[\gn(G)\geq |V(G)|-\kappa_f(G),\]
and for some positive integer $s \geq 2$, there is a guessing strategy $\mc{F}$ on $(G,s)$ such that
\[|V(G)|-\kappa_f(G) = |V(G)|+\log_s \mb{P}[\Win(G,s,\mc{F})].\]
\end{theorem}

\begin{remark}
In \cite{Christofides&Markstrom11} it was proved that the above lower bound is actually an equality for various families of undirected graphs including perfect graphs, odd cycles and complements of odd cycles. This led to a conjecture that the inequality is actually an equality, but this conjecture was proved to be false in \cite{Baber13A}. The counter-example is a graph on $10$ vertices which contains many cliques of size $3$.
\end{remark}

\section{Triangle-free graphs with large guessing number}\label{Sec:TriangleFreeGraphs}
In view of the last remark, a natural question is: if we forbid the appearance of triangles in an undirected graph, then is the fractional clique cover the best guessing strategy for our undirected graphs? In other words, 
\begin{conjecture} \label{Conj:LowerBoundSharpTriagFree}
If $G$ is an undirected triangle-free graph then
\[\gn(G) = |V(G)|-\kappa_f(G).\]
\end{conjecture}

A useful bound on $\kappa_f(G)$ which we will make use of is given by the following lemma.
\begin{lemma}\label{Lem:BoundOnFractionalClique}
For any undirected graph $G$
\[
\kappa_f(G)\geq \frac{|V(G)|}{\omega(G)},
\]
where $\omega(G)$ is the number of vertices in a maximum clique in
$G$.
\end{lemma}

\begin{proof}
Let $w$ be an optimal regular fractional clique cover. Since
$\sum_{k\in K(G,v)} w(k) = 1$ holds for all $v\in V(G)$, summing
both sides over $v$ gives us,
\[
\sum_{k\in K(G)} w(k)|V(k)| = |V(G)|,
\]
where $|V(k)|$ is the number of vertices in clique $k$. The result
trivially follows from observing
\[
\sum_{k\in K(G)} w(k)|V(k)| \leq \sum_{k\in K(G)} w(k)\omega(G) =
\kappa_f(G) \omega(G). \qedhere
\]
\end{proof}

\begin{corollary}\label{Cor:KappaTriangleFree}
For triangle-free graph $G$, the $\kappa_f(G) \geq |V(G)|/2$.
\end{corollary}

We will show in this section that there are triangle-free graphs for which the asymptotic guessing $\gn(G)$ of $G$ is strictly greater than $|V(G)|/2$. Combining this with Corollary \ref{Cor:KappaTriangleFree}, we will prove that the answer to the Conjecture \ref{Conj:LowerBoundSharpTriagFree} is negative. Before illustrating our results, we need to introduce the following definition:
\begin{definition} \label{Def:MatrixRepresentsGraph}
Given graph $G = (V, E)$ of order $n$, we say a square matrix $M$ of order $n$ with entries selected from a finite field $\F_q$ of $q$ elements with rows and columns indexed by vertices $i \in V(G)$ represents $G$ over $\F_q$ if the diagonal entries of $M$ are non-zero and the non-diagonal entries $m_{ij}$ are $0$ whenever $ij \not\in E(G)$.
\end{definition}

Let $M$ be a representing matrix of $G$ over $\F_q$. We can form a guessing strategy for $(G, q)$ by asking each player $i$ to adapt an assumption that the assigned values of itself and every player in its neighbourhood are taken from $\F_q$ and satisfy a linear equation 
\[m_{ii} x_i + \sum_{j \in \Gamma(i)} m_{ij} x_j = 0\]
where $x_i$, $x_j$s are assigned values of players $i$, $j$s, and coefficients $m_{ii}$ and $m_{ij}$ are the $(i,i)$-th and $(i,j)$-th entries of $M$. Then the value of $x_i$ produced by this strategy is
\[x_i=m_{ii}^{-1}\sum_{j \in \Gamma(i)} m_{ij} x_j,\]
which is well-defined since $m_i\ne0$ by assumption.

The guessing game $(G, q)$ is won by adopting strategy $M$ if the assigned values $X = \begin{pmatrix} x_1 & x_2 & \cdots & x_n \end{pmatrix}^\top$ give a solution of a system of linear equations $MX = 0$ defined over $\F_q$.
\begin{equation} \label{Equ:GuessingNumberofMatrix}
\gn(G, q, M) = \log_q |\{X \in \F_q^n| MX = 0\}| = n - \rk_{\F_q}(M)
\end{equation}
We note that $MX = 0$ always has a trivial solution $X = \mb{0}$ hence $\log_q |\{X \in \F_q^n| MX = 0\}|$ is well defined.

The value $n - \rk_{\F_q}(M)$ is a valid lower-bound of $\gn(G)$, i.e.
\[\gn(G) \geq \gn(G, q, M) = n - \rk_{\F_q}(M).\]
It is clear that we can disprove Conjecture \ref{Conj:LowerBoundSharpTriagFree} by constructing a triangle-free graph that has a representation matrix $M$ with $\rk_{\F_q}(M) < |V(G)|/2$ over some finite field $\F_q$.

\begin{definition}\label{Def:SteinerSystem}
A \emph{Steiner system} $S(t, k, n)$ is a family of $k$-element subsets of $\{1, 2, \ldots, n\} =: [n]$ with the property that each $t$-element subset of $[n]$ contained in exactly one element of $S(t, k, n)$. Elements of $S(t, k, n)$ are called blocks, and elements of $[n]$ are referred to as points.
\end{definition}

For more information about Steiner systems, and the particular system used here, we refer to \cite[Chapter 1]{CavL}.

The following proposition plays a crucial role in our construction:

\begin{proposition}
The Steiner system $S(3, 6, 22)$ has the following properties:
\begin{itemize}[font=\normalfont]
\item[(a)] $S(3, 6, 22)$ contains $77$ blocks.
\item[(b)] Any two blocks in $S(3, 6, 22)$ intersect in zero or two points.
\item[(c)] No three blocks in $S(3, 6, 22)$ are disjoint.
\item[(d)] Each point is contained in exactly $21$ blocks.
\end{itemize}
\end{proposition}
\begin{proof}
(a) We simply count the number of blocks containing a fixed set of points.
Given two points $i, j$ in $[n]$, there are $20$ choices of the third point $k \in [n] \backslash \{i,j\}$ to form a group of $3$ points. By definition, any $3$ points of $[n]$ belongs to exactly one block, hence there are exactly $20$ blocks containing both two fixed points $i$ and $j$. 

Let $B$ and $C$ be two blocks containing both $i$ and $j$. We have $B \cap C = \{i, j\}$ and there are $4$ points in $B$ other than $i$ and $j$, so there are $20/4 = 5$ blocks that contain both $i$ and $j$.

Now we fix one point $i$ in $[n]$. There are $21$ pairs of $[n]$ containing $i$ and if $x$ is a block that contains $i$ then it also contains $5$ pairs of $[n]$ containing $i$, hence each point $i$ of $[n]$ belongs to $21\cdot5/5 = 21$ blocks.

We repeat our argument for zero point of $[n]$ and we derive that there are $22\cdot21/6 = 77$ blocks of $S(3, 6, 22)$.

(b) We see in the first part that each point in $S(3, 6, 22)$ belongs to $21$ different blocks. If we fix a point $p$, then there are $21$ points $q \neq p$, and each of these points belongs to $5$ blocks that containing $p$. The system of $21$ blocks on $21$ points satisfies the following properties:
\begin{itemize}
\item[(i)] For every two distinct points $q$, $l \neq p$, there is exactly one block that contains both points (by definition of $S(3, 6, 22)$).
\item[(ii)] Let $B$ be a block in the set of $21$ blocks containing $p$. For each point $q \neq p$ in $B$, there are exactly $5$ blocks contains $q$ including $B$ (from (i)). Moreover, it is clear that for any two distinct points $q$, $l$ which are different from $p$, the set of blocks containing $q$ and the set of blocks containing $l$ share $B$ as their unique common element. Since $B$ is arbitrarily, this shows that for any two blocks $B$ and $C$ in the set of $21$ blocks containing a fixed point $p$, $B$ and $C$ intersect at exactly one point beside $p$.
\item[(iii)] Let $B$ be a block that contains the fixed point $p$. We choose other $3$ points $q$, $k$, $l$ in $B$ and a point $h$ that does not belong to $B$. The set of four points $\{q, k, l, h\}$ obviously cannot be contained in one single block of the $21$ blocks having $p$ as their element.
\end{itemize}
Hence these $21$ blocks on $21$ points form a projective plane where each block is a line in this plane. A corollary is that any two blocks that contain a fixed point $p$ must also contain another point $q$. This proves that any two blocks in $S(3, 6, 22)$ either intersect in zero or two points.

(c) We fix a block $B \in S(3, 6, 22)$. There are $16$ points that are not in $B$. Moreover, for every two distinct pairs of points of $B$, the set of blocks containing one pair and the set of blocks containing the other pair share $B$ as their unique common element. Therefore, there are exactly $60$ blocks having non-empty intersection with $B$. This leaves $16$ blocks that have empty intersection with $B$. From (b) we know that any two blocks must intersect in zero or two points, this makes the $16$ points and $16$ blocks a symmetric balanced incomplete block design (BIBD) $(16, 6, 2)$. It follows from the property of symmetric BIBD that any two blocks intersect in $2$ points.

(d) This is already proved in part (a).
\end{proof}

\begin{theorem} \label{Thm:TriagFree}
There exists an undirected triangle-free graph $G$ on $100$ vertices with $\gn(G) \geq 77$.
\end{theorem}

\begin{proof}
We define the vertex set of the graph $G$ to be $22$ points plus $77$ blocks of the Steiner system $S(3,6,22)$ plus an extra point $\{\infty\}$. We define an edge between two vertices $u$ and $v$ if one of the following conditions is satisfied:
\begin{itemize}
\item $u$ is $\{\infty\}$, and $v$ is a point.
\item $u$ is a point and $v$ is a block which contains $u$ as an element.
\item $u$ and $v$ are blocks of $S(3,6,22)$ and $u \cap v = \emptyset$.
\end{itemize}

According to the previous proposition, the graph obtained form our construction is triangle-free. It remains to show that there is a matrix representing $G$ with rank less than $50$ over some finite field $\F_q$.

The chosen matrix is $A+I$ where $A$ is the adjacency matrix of $G$ and $I$ is the identity matrix of order $100$. The rank of the matrix $A+I$ is $23$ over the finite field $\F_3$ (see the next lemma).

In this graph, the size of the maximal independent set is $22$ (and the independent sets of size $22$ are the vertex neighbourhoods), hence the guessing number of this graph is at most $78$.
\end{proof}

The constructed graph is in fact the Higman--Sims graph \cite{HiSi}, which is a strongly regular triangle-free graph with parameters $(100,22,6)$. The Higman-Sims graph was first introduced by Dale Mesner in his 1956 PhD thesis \cite{Mesner}; see \cite{KlWo} for a historical account.

\begin{proposition}
If $A$ is the adjacency matrix of the Higman--Sims graph, then the rank of
$A+I$ over $\F_3$ is $23$.
\end{proposition}

\begin{proof}
Let $r_v$ be the row of $B = A + I$ corresponding to vertex $v$. We write the
vertex set as $\{\infty\}\cup X\cup Y$, where $X$ and $Y$ are the neighbours
and non-neighbours of $\infty$. Consider the 22 vectors $r_x$ for
$x\in X$. Since the graph is triangle-free, the restriction of $r_x$ to
the coordinates in $X$ has a one in position $x$ and zeros elsewhere; so
these $22$ vectors are linearly independent. Take the $23$rd vector to be
the all-1 vector $j$. Note that $j$ is not in the span of the first $22$.
For if it were, it would have to be their sum (looking at the restriction
to $X$. But the sum of the $r_x$ has coordinate $22 \equiv 1\mod 3$ at $\infty$,
$1$ at each point of $X$, and $6 \equiv 0\mod 3$ at each point of $Y$; that is, it is
$r_\infty$. So our $23$ vectors are linearly independent. Also, they are all
contained in the row space of $B$. (This is clear for the $r_x$; also the
sum of all the vectors $r_v$ is $2j$, since all column sums of $B$ are $23 \equiv 2\mod 3$,
so $j$ is also in the row space.

We claim that they span the row space. It is clear that their span contains
all $r_x$ for $x\in X$, and we just showed that it contains $r_\infty$.
Take a vertex $y\in Y$. Consider the sum of the vectors $r_x$ for the $16$
vertices $x\in X$ which are not joined to $y$. This has coordinate $16 \equiv 1\mod 3$
at $\infty$, $0$ at points of $X$ joined to $y$, and $1$ at points of $X$
not joined to $y$. The coordinate at $y$ is zero. If $y'$ is joined to $y$,
then the six neighbours of $y'$ in $X$ are a subset of the 16 points not
joined to $y$, so the coefficient at $y'$ is $6 \equiv 0\mod 3$. If $y'$ is not joined
to $y$, then $y'$ is joined to two neighbours of $y$ in $X$ and to four
non-neighbours, so the coefficient at $y'$ is $4 \equiv 1\mod 3$. Thus the sum of our
sixteen vectors is $j-r_y$, showing that $r_y$ lies in the span of our
$23$ chosen vectors.

Notice incidentally that $B^2 = 2B$, so that the minimum polynomial of $B$
is the product of distinct linear factors, so $B$ is diagonalisable.
\end{proof}

We also found other strongly regular triangle-free graphs which have guessing number larger than the lower bound given by fractional clique cover. See \cite[Chapter 8]{CavL} for further details about these graphs. 

\begin{proposition}
The following triangle-free graphs on $n$ vertices have their guessing number larger than $n/2$:
\begin{itemize}[font=\normalfont]\itemsep0pt
\item[(a)] The Clebsch graph on $16$ vertices has $10 \leq \gn(G) \leq 11$.
\item[(b)] The Hoffman--Singleton graph on $50$ vertices has $29 \leq \gn(G) \leq 35$.
\item[(c)] The Gewirtz graph on $56$ vertices has $36 \leq \gn(G) \leq 40$.
\item[(d)] The $M_{22}$ graph on $77$ vertices has $55 \leq \gn(G)\leq 56$.
\item[(e)] The Higman-Sims graph on $100$ vertices has $77 \leq \gn(G) \leq 78$.
\end{itemize}
\end{proposition}

\begin{proof}

The upper bound for the guessing number derived for these graphs is $\gn(G) \leq |V(G)| - \alpha(G)$ where $\alpha(G)$ is the independence number \cite{Riis07}. The independence numbers of the Clebsch graph, the Hoffman--Singleton graph, and the Gewirtz graph are $5$, $15$, and $16$, respectively \cite{Brouwer}. 
The independence number of the $M_{22}$ graph is $21$, as shown below.

(a) The Clebsch graph is a triangle-free strongly regular graph with parameters $(16,5,2)$ which is constructed as follows:
We start with a finite set $S = \{1, 2, 3, 4, 5\}$. The set $V$ contains all subsets of size $1$, and $2$ of $S$, and $V$ also contains an extra single point set $\{*\}$. We form the $(16,5,2)$ graph with vertex set $V$ and an edge between two vertices $u$ and $v$ if one of the following conditions is satisfied:
\begin{itemize}
\item $u$ is $\{\infty\}$, and $v$ is a subset of $S$ with cardinal $1$.
\item $u$ is a subset of $S$ with cardinal $1$, $v$ is a subset of $S$ with cardinal $2$, and $u$ is a subset of $v$.
\item $u$ and $v$ are subsets of $S$ with cardinal $2$, and $u$ intersect $v$ is empty.
\end{itemize}
We have the rank of the matrix $A+I$ is $6$ over finite field $\F_2$. A basis of $A+I$ is $\{j\} \cap \{r_v| v \in \Gamma(\infty)\}$, where $r_v$ is the row of $A+I$ corresponding to vertex $v$ and $j$ is the all--1 vector.

(b) The Hoffman-Singleton graph which is triangle-free strongly regular with parameters $(50, 7, 1)$ has one way of construction as follows:
We take five 5-cycles $C_h$ and their complements $C^c_i$, and we join vertex $j$ of $C_h$ to vertex $hi+j \mod 5$ of $C^c_i$. This construction is due to Conway. The rank of the matrix $A+3I$ over finite field $\F_5$ is $21$\footnote{Brouwer and Van Eijl derived the same result for $A+3I$ over $\F_5$ \cite[pages 340, 341]{Brouwer&VanEijl} using eigenvalue method.}. A basis for this matrix over $\F_5$ is recorded in the file \url{Basis.txt} which can be downloaded from \url{https://www.eecs.qmul.ac.uk/~smriis/}. This file also includes a description for coordinates of each row in $A+3I$ over $\F_5$ with respect to the given basis.

(c) The Gewirtz graph with parameters $(56, 10, 2)$ can be constructed from the $S(3,6,22)$ by fixing an element and let the vertices be the $56$ blocks not containing that element. Two vertices are adjacent if the intersection of their corresponding blocks is empty. The rank of the matrix $A+I$ over finite field $\F_3$ is $20$. A basis for this matrix over $\F_3$ is recorded in the file \url{Basis.txt}.

(d) The triangle-free strongly regular graph $M_{22}$ with parameters $(77, 16, 4)$ which can be constructed by let the $77$ blocks of $S(3,6,22)$ be the vertices of the graph, and an edge $uv$ between two vertices $u$ and $v$ if $u$ and $v$ are disjoint as blocks. Note that this is the induced subgraph of the Higman--Sims graph on the set of non-neighbours of $\infty$.

To see that its independence number is $21$, note that the vertices other than $\infty$ non-adjacent to a vertex in $X$ in the Higman--Sims graph form an independent set of size~$21$; and there is no larger independent set, since all independent sets of size $22$ in the Higman--Sims graph are vertex neighbourhoods.

The rank of $A+I$ over finite field $\F_3$ is $22$. A basis for this matrix over $\F_3$ is recorded in the file \url{Basis.txt}.

(e) Theorem \ref{Thm:TriagFree}.
\end{proof}

\begin{remark}
For a more extensive list of computations of ranks of matrices $A + kI$ over $\F_q$ for $q = 2, 3, 5, 7$ see \url{EBasis.zip} at \url{https://www.eecs.qmul.ac.uk/~smriis/}.
\end{remark}

\begin{problem}
Find the exact value of the guessing number of the strongly regular graphs considered here.
\end{problem}

\end{document}